\documentclass[a4paper,12pt]{article}
\usepackage{}
\usepackage{mathrsfs}
\usepackage{amsfonts}
\usepackage{amscd,color}
\usepackage{amsmath,amsfonts,amssymb,amscd}
\usepackage{indentfirst,graphicx,epsfig}
\usepackage{graphicx,psfrag}
\input{epsf}

\newtheorem{df}{Definition}[section]
\newtheorem{thm}{Theorem}[section]

\newtheorem{lem}{Lemma}[section]

\newtheorem{claim}{Claim}[section]

\newtheorem{cor}{Corollary}[section]

\newenvironment {proof} {\noindent{\em Proof.}}{\hspace*{\fill}$\Box$\par\vspace{4mm}}

\def\qed{\hfill \nopagebreak\rule{5pt}{8pt}}

\title{Proper connection number of random graphs\footnote{Supported by NSFC No.11371205, ``973" program No.2013CB834204, and PCSIRT.}}

\author{Ran Gu, Xueliang Li, Zhongmei Qin\\
{\small  Center for Combinatorics and LPMC-TJKLC}\\
{\small Nankai University, Tianjin 300071, P.R. China}\\
{\small Email: guran323@163.com, lxl@nankai.edu.cn, qinzhongmei90@163.com}\\
}
\date{}
\begin{document}
\maketitle
\begin{abstract}
A path in an edge-colored graph is called a proper path if no two adjacent edges of the path are colored the same. For a connected graph $G$, the proper connection number $pc(G)$ of $G$ is defined as the minimum number of colors needed to color its edges, so that every pair of distinct vertices of $G$ is connected by at least one proper path in $G$. In this paper, we show that almost all graphs have the proper connection number 2. More precisely, let $G(n,p)$ denote the Erd\"{o}s-R\'{e}nyi random graph model, in which each of the $\binom{n}{2}$ pairs of vertices appears as an edge with probability $p$ independent from other pairs. We  prove that for sufficiently large $n$, $pc(G(n,p))\le2$ if $p\ge\frac{\log n +\alpha(n)}{n}$, where $\alpha(n)\rightarrow \infty$.  \\[2mm]
\end{abstract}

\section{Introduction}

All graphs in this paper are undirected, finite and simple. We follow \cite{BM} for graph theoretical notation and terminology not defined here. Let $G$ be a nontrivial connected graph with an {\it edge-coloring} $c : E(G)\rightarrow \{1, 2, \ldots, t\},
\ t \in \mathbb{N}$, where adjacent edges may have the same color. A path of $G$ is called a {\it rainbow path} if no two edges on the path have the same color. The graph $G$ is called {\it rainbow connected} if for any two vertices of $G$ there is a rainbow path of $G$ connecting them. An edge-coloring of a connected graph is called a {\it rainbow connecting coloring} if it makes the graph rainbow connected. For a connected graph $G$, the \emph{rainbow connection number} $rc(G)$ of $G$ is the smallest number of colors that are needed in order to make $G$ rainbow connected.   This concept of rainbow connection of graphs was introduced by Chartrand et al. \cite{CJMZ} in 2008. The interested readers can see \cite{LS,LSS} for a survey on this topic.

Motivated by rainbow coloring and proper coloring in graphs, Andrews et al. \cite{ALLZ} introduced the concept of proper-path coloring. Let $G$ be a nontrivial connected graph with an edge-coloring. A path in $G$ is called a \emph{proper path} if no two adjacent edges of the path are colored the same. An edge-coloring of a connected graph $G$ is a \emph{proper-path coloring} if every pair of distinct
vertices of $G$ are connected by a proper path in $G$. For a connected graph $G$, the minimum number of colors that are needed to produce a proper-path coloring of $G$ is called the \emph{proper connection number} of $G$, denoted by $pc(G)$. From the definition, it follows that $1\le pc(G)\le min\{ rc(G), \chi'(G)\}\le m$, where $\chi'(G)$ is the chromatic index of $G$ and $m$ is the number of edges of $G$. And it is easy to check that $pc(G) = 1$ if and only if $G = K_n$, and $pc(G) = m$ if and only if $G = K_{1,m}$. For more details we refer to \cite{ALLZ, BFGMMMT}.

The study on rainbow connectivity of random graphs has attracted the interest of many researchers, see \cite{Caro,FT,HL}. It is worth   investigating the proper connection number of random graphs, which is  the purpose of this paper. The most frequently occurring probability model of random graphs is the Erd\"{o}s-R\'{e}nyi random graph model $G(n,p)$ \cite{ER}. The model $G(n,p)$ consists of all graphs with
$n$ vertices in which the edges are chosen independently and with
probability $p$. We say an event $\mathcal{A}$ happens
\textit{with high probability} if the probability that it happens approaches $1$ as $n\rightarrow \infty $, i.e., $Pr[\mathcal{A}]=1-o_n(1)$. Sometimes, we say \textit{w.h.p.} for short.
We will always assume that $n$ is the variable that tends to infinity.

Let $G$ and $H$ be two graphs on $n$ vertices. A property $P$ is said to be \emph{monotone} if whenever $G\subseteq H$ and $G$ satisfies $P$, then $H$ also satisfies $P$. For any property $P$ of graphs and
any positive integer $n$, define $Prob(P, n)$ to be the ratio of the number of graphs with $n$ labeled vertices having $P$ divided by the total number of graphs with these vertices. If $Prob(P, n)$ approaches 1 as $n$ tends to infinity, then we say that almost all graphs have the property $P$. Similarly, for a fixed integer $r$, we say that almost all $r$-regular graphs have the property $P$ if the ratio of the number
of $r$-regular graphs with $n$ labeled vertices having $P$ divided by the total number of $r$-regular graphs with these vertices approaches 1 as $n$ tends to infinity.

There are many results in the literature asserting that
almost all graphs have some property. Here we list some of them, which are related to our study on the proper connection number of random graphs.
\begin{thm}\label{thad}\cite{BH}
Almost all graphs are connected with diameter 2.
\end{thm}

\begin{thm}\label{thac}\cite{BH}
For every nonnegative integer $k$, almost all graphs are $k$-connected.
\end{thm}

\begin{thm}\label{thar}\cite{RW}
For fixed integer $r\ge 3$, almost all $r$-regular graphs are Hamiltonian.
\end{thm}

In \cite{BFGMMMT}, Borozan et al. got the following result.
\begin{thm}\label{thd2}
If the diameter of graph $G$ is 2 and $G$ is 2-connected, then $pc(G)=2$.
\end{thm}

The authors in \cite{ALLZ} proved the following result.
\begin{thm}\label{th0}
If $G$ is not complete and has a Hamiltonian path, then $pc(G)=2$.
\end{thm}

From Theorem \ref{thad} and Theorem \ref{thac} and the formula that $Pr[A\cap B]=Pr[A]+Pr[B]-Pr[A\cup B]$, it is easy to derive that almost all graphs are 2-connected with diameter 2. Hence, by Theorem \ref{thd2}, we have
\begin{thm}\label{thap}
Almost all graphs have the proper connection number 2.
\end{thm}

Even if we concentrate on regular graphs, from Theorem \ref{thar} and  Theorem \ref{th0}, we also have the following result.
\begin{thm}
For fixed integer $r\ge 3$, almost all $r$-regular graphs have the proper connection number 2.
\end{thm}

Next, we study the value of the proper connection number of $G(n,p)$, when $p$ belongs to different ranges.
The following theorem is a classical result on the connectedness of a random graph.
\begin{thm}\label{ppt}\cite{ER}
Let $p=(\log n +a)/n$. Then
\begin{equation*}
Pr[G(n,p)\  is \ connected]\rightarrow
\left\{
  \begin{array}{ll}
   e^{-e^{-a}} & \hbox{ if $|a|=O(1)$,} \\
    0 & \hbox{ if  $a\rightarrow -\infty$}, \\
    1 & \hbox{ if $a\rightarrow +\infty$.}
  \end{array}
\right.
\end{equation*}
\end{thm}

Since the concept of proper-path coloring only makes sense when the  graph is connected, we only study on the proper-path coloring of $G(n,p)$ which is w.h.p. connected. Our main result is as follows.
\begin{thm}\label{thm1}
For sufficiently large $n$, $pc(G(n,p))\le2$ if $p\ge\frac{\log n +\alpha(n)}{n}$, where $\alpha(n)\rightarrow \infty$.
\end{thm}

We prove Theorem \ref{thm1} in Section 2. In Section 3, we give some results on the proper connection number of general graphs.

\section{Proof of Theorem \ref{thm1}}

In order to prove the first part of Theorem \ref{thm1}, we first present a classical result on random graphs as follows.
\begin{thm}\cite{B}\label{thH}
Let $\omega(n)\rightarrow \infty$, $p=\frac{1}{n}\{\log n+\log\log n+\omega(n)\}$. Then, w.h.p. $G(n,p)$ is Hamiltonian.
\end{thm}

Let $p'=\frac{1}{n}\{\log n+\log\log n+\omega(n)\}$, where $\omega(n)\rightarrow \infty$. Since Hamiltonian is a monotone property, combining with Theorem \ref{th0}, we know that $pc(G(n,p))=2$ if $p'\leq p<1$. Thus in the sequel, we assume that $p=\frac{\log n +\alpha(n)}{n}$, where $\alpha(n)=o(\log n)$, and $\alpha(n)\rightarrow \infty$.

For two disjoint vertex-subsets $X$ and $Y$ of $G$, let $e(X,Y)$ be the number of the edges with one endpoint in $X$ and the other in $Y$. For vertex-subsets $U\subset S$, $N(U,S)$ is the disjoint neighbor set of $U$ in $G[S]$, i.e., $N(U,S)=\{w\in S-U:\ \exists u\in S\ and \ \{uw\}\in G[S]\}$ and $d_S(v)=|N(v)\cap S|$ is the degree of $v$ in $S$. For ease of notation, let $G\in G(n,p)$ and denote by $V$ the vertex set of $G(n,p)$.

It is known that w.h.p. the diameter of $G(n,p)$ is asymptotically equal to $D=\frac{\log n}{\log \log n}$ \cite{B}. We call a vertex $u$ \emph{large} if its degree $d(u)\geq \frac{\log n}{100}$ and \emph{small} otherwise. Let SMALL denote the vertex-subset consisting of all the small vertices. We first give some properties of small vertices as follows.
\begin{lem}\label{lem1}
The following hold w.h.p. in $G(n,p)$.
\begin{itemize}
  \item [(1)]$|SMALL|\le n^{0.1}$.
  \item [(2)]No pair of small vertices are adjacent or share a common neighbor.
\end{itemize}
\end{lem}
\begin{proof}
(1) Let $s=\lceil n^{0.1}\rceil$. Let $\mathcal{A}$ denote the event that there exists a vertex-subset $S$ with order $s$ such that each vertex $v\in S$ is small. Then $\mathcal{A}$ happens with probability
\begin{align*}
Pr[\mathcal{A}]& \le\binom{n}{s}\left[\sum\limits_{k = 0}^{\frac{{\log n}}{{100}}} {\left( {\begin{array}{*{20}{c}}
n\\
k
\end{array}} \right)} {p^k}{(1 - p)^{n - 1 - k}}\right]^s\\
                & \le  {\left( {\frac{{ne}}{s}} \right)^s}{\left[ {\frac{{\log n}}{{100}}{{\left( {\frac{{100ne}}{{\log n}}} \right)}^{\frac{{\log n}}{{100}}}}{{\left( {\frac{{\log n + \alpha (n)}}{n}} \right)}^{\frac{{\log n}}{{100}}}}{e^{ - \frac{{\log n + \alpha (n)}}{n}\left( {n - 1 - \frac{{\log n}}{{100}}} \right)}}} \right]^s}  \\
                &\le  {\left( {\frac{{ne}}{s} \cdot \frac{{\log n}}{{100}}{{\left( {101e} \right)}^{\frac{{\log n}}{{100}}}}{e^{ - (\log n + \alpha (n) )+ \frac{{\log n + \alpha (n)}}{n} + \frac{{\log n}}{{100}} \cdot \frac{{\log n + \alpha (n)}}{n}}}} \right)^s}\\
                &\le{\left( {\frac{{ne}}{s} \cdot \frac{{\log n}}{{100}} \cdot {n^{\frac{6}{{100}}}} \cdot {n^{ - 1}} \cdot O(1)} \right)^s}\\
                &\le O(n^{-0.01\cdot s}).
\end{align*}
That implies that w.h.p. $|SMALL|\le n^{0.1}$.

(2) Let $\mathcal{B}$ denote the event that there exist two small vertices $x$, $y$ and the distance between $x$ and $y$ is at most 2.
We have
\begin{align*}
Pr[\mathcal{B}]&\le \binom{n}{2}\Bigg\{p\left(\sum\limits_{i = 1}^{\frac{{\log n}}{{100}} } {\left( {\begin{array}{*{20}{c}}
{n - 2}\\
i
\end{array}} \right)} {p^i}{\left( {1 - p} \right)^{n - 2 - i}}\right)^2\\
                &\null+\binom{n-2}{1}p^2\left(\sum\limits_{i = 1}^{\frac{{\log n}}{{100}} } {\left( {\begin{array}{*{20}{c}}
{n - 3}\\
i
\end{array}} \right)} {p^i}{\left( {1 - p} \right)^{n - 3 - i}}\right)^2\Bigg\}\\
                                 &\le  n^2\left[\frac{\log n +\alpha(n)}{n}+n\left(\frac{\log n +\alpha(n)}{n}\right)^2\right]
                                 \left[2\binom{n}{\frac{{\log n}}{{100}}}p^{\frac{{\log n}}{{100}}}(1-p)^{n-2-\frac{{\log n}}{{100}}}\right]^2\\
                                 &\le\left[n(2\log n)+n(2\log n)^2\right]\left[2\left(\frac{ne}
                                 {\frac{{\log n}}{{100}}}\right)^{\frac{{\log n}}{{100}}}p^{\frac{{\log n}}{{100}}}\left(1-p
                                 \right)^{n-2-\frac{{\log n}}{{100}}}\right]^2\\
                                 &\le\left[n(2\log n)+n(2\log n)^2\right]n^{-1.9}\\
                                 &\le n^{-0.8}.
\end{align*}
\end{proof}

From Lemma \ref{lem1}, we can obtain that every small vertex is adjacent to a large vertex and there is at most one small vertex among the neighbors of a large vertex. Thus, we can find a matching $M$ consisting of $|SMALL|$ edges in $G$ such that for every edge $e$ in $M$, one endpoint of $e$ is small and the other endpoint is large. Let $s=|M|=|SMALL|$. Denote the large vertices in $M$ by $x_1,x_2,\ldots,x_s$ and denote the small vertices in $M$ by $y_1,y_2,\ldots,y_s$. Without loss of generality, we assume that for every $i\in \{1,2,\cdots, s\}$, $\{x_iy_i\}$ is an edge in $M$. If $|V\backslash SMALL|$ is odd, then we take an arbitrary edge $\{uv\}$ disjoint from $M$ and let $M'=M\cup \{uv\}$. If $|V\backslash SMALL|$ is even, just let $M'=M$. Denote the cardinality of $M'$ by $s'$, that is, \begin{equation*}
s'=
\left\{
  \begin{array}{ll}
  s & \hbox{ if $|V\backslash SMALL|$ is even,} \\
    s+1 & \hbox{  if $|V\backslash SMALL|$ is odd.}
  \end{array}
\right.
\end{equation*}
Let
\begin{equation*}
V_1=
\left\{
  \begin{array}{ll}
 V\backslash SMALL & \hbox{ if $|V\backslash SMALL|$ is even,} \\
    V\backslash (SMALL\cup \{u\} )& \hbox{  if $|V\backslash SMALL|$ is odd.}
  \end{array}
\right.
\end{equation*}
So $|V_1|$ is even.

The following is an important structural property of $G$.

\begin{claim}\label{claim1}
The induced subgraph $G[V_1]$ of $G$ is w.h.p. Hamiltonian.
\end{claim}

Note that to prove $pc(G) \le 2$, it suffices to give $G$ an edge-coloring with 2 colors and verify that the edge-coloring is a proper-path coloring of $G$. Denote the Hamiltonian cycle of $G[V_1]$ by $C$. We color the edges of $C$ consecutively and alternately with color 1 and 2, and color all the edges in $M'$ with color 1. It is easy to get that under this partial coloring, every pair of large vertices have a proper path connecting them, and there exists a proper path connecting a  vertex in $\{y_1,y_2,\ldots,y_s,u\}$ (if such $u$ exists) with a vertex in $V_1$. The following claim helps us to take care of pairs of vertices in $\{y_1,y_2,\ldots,y_s,u\}$ .

\begin{claim}\label{claim2}
There exists an edge-coloring of edges in $E(G)\backslash (E(C)\cup M')$ with 2 colors such that w.h.p. every pair of vertices in $\{y_1,y_2,\ldots,y_s,u\}$  have a proper path connecting them in $G$.
\end{claim}

Thus Theorem \ref{thm1} follows from the above arguments.
So all we need to do is to prove Claims \ref{claim1} and \ref{claim2}.

\subsection{Proof of Claim \ref{claim1}}

We will use the similar arguments of Cooper \cite{CF} and Frieze \cite{FK}. The following lemma establishes some structural properties of $G$, which we will make use of in our proof.

\begin{lem}\label{lem2}
The following hold in $G$ w.h.p. :
\begin{itemize}
  \item [(1)]For any $S \subseteq V$, $|S|\le \frac{n}{375}$ implies $|E(G[S])|< \frac{|S|np}{250}$.
  \item [(2)]If $U, W \subseteq V$, $U\cap W= \emptyset$, $|U|, |W| \ge \frac{n}{\log\log n}$, then $e(U, W)>0$.
  \item [(3)]There are at most $n^{0.2}$ edges incident with vertices in SMALL.
\end{itemize}
\end{lem}
\begin{proof}
(1) The number of edges in an induced subgraph $G[S]$ with $|S|=s$ is a binomial random variable with parameters $\binom{s}{2}$ and $p$. By Bollob\'{a}s \cite{B} we have for large deviations of binomial random variables
$$Pr[the \ number\  of\ edges\ in\ G[S]\ge\gamma \binom{s}{2}p ]<\left(\frac{e}{\gamma}\right)^{\gamma \binom{s}{2}p}.$$ Setting $\gamma=\frac{n}{125s},$ we obtain that
\[\sum\limits_{s = 1}^{\frac{n}{{375}}} {\left( {\begin{array}{*{20}{c}}
n\\
s
\end{array}} \right)}\left(\frac{e}{\gamma}\right)^{\gamma \binom{s}{2}p}=o(1). \]

(2) Let $\mathcal {A}$ denote the event that there exist two subsets $U, W \subseteq V$, $U\cap W= \emptyset$, $|U|, |W| \ge \frac{n}{\log\log n}$ and $e(U, W)=0$. Then
\begin{align*}
Pr[\mathcal {A}]&\le \sum_{s \ge \frac{n}{\log \log n}}\sum_{t \ge \frac{n}{\log \log n}}\binom{n}{s}\binom{n-s}{t}(1-p)^{st}\\
&\le  \sum_{s \ge \frac{n}{\log \log n}}\sum_{t \ge \frac{n}{\log \log n}}\binom{n}{s}\binom{n-s}{t}e^{-\frac{\log n+\alpha(n)}{n}st}\\
&\le \sum_{s \ge \frac{n}{\log \log n}}\sum_{t \ge \frac{n}{\log \log n}}\binom{n}{s}\binom{n-s}{t}e^{-\frac{\log n}{n}\cdot \frac{n}{\log \log n}\cdot \frac{n}{\log \log n}}\\
&\le o(n^{-1}).
\end{align*}

(3) Since SMALL is an independent set, i.e., no edges in the induced subgraph $G[SAMLL]$, we have that the number of edges incident to SMALL is no more than
$$|SMALL| \cdot \frac{\log n}{100} \le n^{0.1} \cdot \frac{\log n}{100}< n^{0.2}. $$
\end{proof}

Let $\mathscr{H}=\{G \in G(n,p)$: the conditions of Lemmas \ref{lem1} and \ref{lem2} hold$\}$. The following lemma is an immediate consequence of Lemma \ref{lem2}(1).

\begin{lem}\label{lem3}
Let $G \in \mathscr{H}$, $U \subseteq S \subset V$, $|U| \le \frac{n}{1500}$, $F \subset E(G[S])$ and $H=(S,F)$. If $U$ is such that the degree of $w$ in $H$ is at least $\frac{\log n}{101}$ for all $w \in U$, then $|N(U, S)| \ge 3|U|$ in $H$.
\end{lem}

We regard the edges in $G$ as initially colored blue, but with the option of recoloring a set $R$ of the edges red. We require the set $R$ of red edges is ``deletable", which is defined as follows.

\begin{df}
\begin{itemize}
  \item [(1)]$R \subseteq E(G)$ is deletable if
  \begin{itemize}
  \item [(i)]$R$ is a matching.
  \item [(ii)]No edge of $R$ is incident with a small vertex.
  \item [(iii)]$|R|=\lceil n^{0.1}\rceil$.
  \end{itemize}
  \item [(2)]Let $G_B[V_1]$ denote the subgraph of $G[V_1]$ induced by blue edges.
  \item [(3)]$N_B(U, V_1)$ denotes the disjoint neighbor set of $U$ in $G_B[V_1]$.
\end{itemize}
\end{df}

\begin{lem}\label{lem4}
Let $G \in \mathscr{H}$ and let $U \subseteq V_1$, $|U| \le \frac{n}{1500}$. Then $|N_B(U, V_1)| \ge 2|U|$.
\end{lem}
\begin{proof}
By Lemma \ref{lem2}(1), each vertex $w \in U$ has at most one neighbor in SMALL. We have $d_{V_1}(w) \ge \frac{\log n}{100}-1-1\ge \frac{\log n}{101}$. From Lemma \ref{lem3}, we obtain that there are at least $3|U|$ neighbors of $U$ in $V_1$. Thus the removal of $\min\{|R|, |U|\}$ deletable edges makes $|N_B(U, V_1)| \ge 2|U|$.
\end{proof}

\begin{lem}\label{lem5}
For $G \in \mathscr{H}$, $G[V_1]$ is connected.
\end{lem}
\begin{proof}
If $G[V_1]$ is not connected, then by Lemma \ref{lem4} the smallest component cannot consist of less than $\frac{n}{1500}$ vertices.

On the other hand, by Lemma \ref{lem2}(2), any two sets of vertices of size at least $\frac{n}{\log\log n}$ must be connected by an edge. So $G[V_1]$ is connected.
\end{proof}

To prove Claim \ref{claim1}, we also need some more definitions and results taken from P\'{o}sa \cite{P} and Frieze \cite{FK}.

\begin{df}
Let $\Gamma=(V,E)$ be a non-Hamiltonian graph with a longest path of length $\ell$. A pair $\{u,v\}\notin E$ is called a hole if add $\{u,v\}$ to $\Gamma$ creates a graph $\Gamma'$ which is Hamiltonian or contains a path longer than $\ell$.
\end{df}

\begin{df}
A graph $\Gamma=(V,E)$ is called a $(k,c)$-expander if $|N(U)|\ge c|U|$ for every subset $U \subseteq V(G)$ of cardinality $|U|\le k$.
\end{df}

\begin{lem}\cite{FK}\label{lem6}
Let $\Gamma$ be a non-Hamiltonian connected $(k,2)$-expander. Then $\Gamma$ has at least $\frac{k^2}{2}$ holes.
\end{lem}

From Lemmas \ref{lem4} and \ref{lem6}, we obtain that $G[V_1]$ is a $(\frac{n}{1500}, 2)$-expander, and it has at least $\frac{1}{2}(\frac{n}{1500})^2$ holes depending only on $G_B[V_1]$. We define the set $\mathscr{F}$ to be those $G \in \mathscr{H}$ for which the subgraph $G[V_1]$ is not Hamiltonian. Our aim is to prove the following result.

\begin{lem}\label{lem7}
$\frac{|\mathscr{F}|}{|G(n,p)|}=o(1)$.
\end{lem}
\begin{proof}
Let $R$ be a set of red edges of $G$ with the property $P$ that

(i) $R$ is deletable,

(ii) $\lambda(G[V_1])=\lambda(G_B[V_1])$,

where $\lambda(H)$ is the length of a longest path in the graph $H$.

Let $\mathscr{C}$ be the set of all red-blue colorings of $\mathscr{F}$ which satisfy $P$. Let $\lambda=\lambda(G[V_1])$, we have $\lambda < |V_1|$. Recall that there are at most $\mu=\lceil n^{0.2}\rceil$ edges incident with small vertices. Set $r=|R|$. Since $R$ is a matching, we can choose it in at least
\begin{align*}
&\frac{1}{r!}(m-\lambda-\mu)(m-\lambda-\mu-2\Delta)
\ldots(m-\lambda-\mu-2(r-1)\Delta)\\
&\ge \frac{1}{r!}(m-|V_1|-\mu)(m-|V_1|-\mu-2\Delta)
\ldots(m-|V_1|-\mu-2(r-1)\Delta)\\
&\ge \frac{(m-|V_1|)^r}{r!}(1-o(1))
\end{align*}
ways, where $m$ is the number of edges in $G$, and $\Delta$ is the maximum degree of $G$. It is known that $\Delta$ is w.h.p. at most $3np$ (see e.g. \cite{B}).

Hence, $$|\mathscr{C}|\ge |\mathscr{F}|\frac{(m-|V_1|)^r}{r!}(1-o(1)).$$

Consider that we fix the blue subgraph. Then, by the definition of holes, we have to avoid replacing at least $\frac{1}{2}(\frac{n}{1500})^2$ edges when adding back the red edges in order to construct a red-blue coloring satisfying property $P$. Thus
$$|\mathscr{C}|\le \binom{\binom{n}{2}}{m-r}\binom{\binom{n}{2}
-(m-r)-\frac{1}{2}(\frac{n}{1500})^2}{r}.$$
It follows that
\begin{align*}
\frac{|\mathscr{F}|}{|G(n,p)|}&\le \frac{\sum\limits_{m=\frac{1}{100}\binom{n}{2}p}^{\binom{n}{2}}
\left[\binom{\binom{n}{2}}{m-r}\binom{\binom{n}{2}-(m-r)
-\frac{1}{2}(\frac{n}{1500})^2}{r}
\Bigg/\frac{(m-|V_1|)^r}{r!}(1-o(1))\right]}{\binom{\binom{n}{2}}{\binom{n}{2}p}}\\
&\le n^2O(e^{-\frac{r}{1500^2}+\frac{nr}{(n-1)\log n}})\\
&=o(n^{-1}).
\end{align*}
\end{proof}

\subsection{Proof of Claim \ref{claim2}}

We still assume that $G \in \mathscr{H}$ which defined in the previous subsection. Recall that a \emph{$t$-ary tree} with a designated root is a tree whose non-leaf vertices all have exactly $t$ children. For any tree $T_w$ rooted at $w$ and any vertex $x\in T_w\backslash \{w\}$, we use $P_{T_w}(w,x)$ to denote the only path from $w$ to $x$ in $T_w$. We say that $x$ is at depth $k$ of $T_w$ if $P_{T_w}(w,x)$ is of length $k$. For any tree $T_w$, denote by $L_w$ the set of leaves of $T_w$.

Let $E_1=E(G[V_1])\backslash E(C)$ and $H=(V_1, E_1)$ be a subgraph of $G$. Remember that $x_1, \ldots, x_s$ are the large vertices in $M$. Let $x_{s+1}=v$ and $y_{s+1}=u$, if $M'=M \cup \{uv\}$. For every $x_i\in\{x_1, x_2, \ldots, x_s,x_{s+1}\}$, we will build vertex-disjoint $\frac{\log n}{101}$-ary trees $T_{x_i}$ of depth $(\frac{1}{2}+\epsilon)D=(\frac{1}{2}+\epsilon)\frac{\log n}{\log \log n}$ in $H$. Hereafter, let $0<\epsilon<1$ be a sufficiently small real constant.

Note that if we successfully build such vertex-disjoint trees, then the number of leaves of each tree $T_{x_i}$ is $|L_{x_i}|=(\frac{\log n}{101})^{(\frac{1}{2}+\epsilon)D}$, for $i=1,2, \ldots, s+1$. Thus, we have
\begin{align*}
&Pr[there\  exist\ distinct\ i,\ j \ such\ that\ e(L_{x_i}, L_{x_j})=0]\\
&\le \binom{s+1}{2}(1-p)^{(\frac{\log n}{101})^{(1+2\epsilon)D}}\\
&\le n^{0.2}e^{-\frac{\log n}{n}(\frac{\log n}{101})^{(1+2\epsilon)D}}\\
&\le n^{0.2}\cdot n^{-n^\epsilon}\le n^{-\frac{1}{2}n^\epsilon}\\
&=o(1).
\end{align*}
Hence, for every $i \neq j$, there exists a path from $x_i$ to $x_j$ of length $(1+2\epsilon)D+1$ (these paths are not necessarily vertex-disjoint). Denote that path by $P_{ij}$. For every tree $T_{x_i}$, we color the edges between the vertices at depth $2\ell-1$ to $2\ell$ with color 2, and color the edges between the vertices at depth $2\ell$ to $2\ell+1$ with color 1, where $\ell=1,2,\ldots, \lfloor \frac{(\frac{1}{2}+\epsilon)D}{2}\rfloor$. Color the edges between each $L_{x_i}$ and $L_{x_j}$ ($i \neq j$) with the color different from the color used in the edges between the vertices at depth $(\frac{1}{2}+\epsilon)D-1$ to $(\frac{1}{2}+\epsilon)D$. That is, if the edges between the vertices at depth $(\frac{1}{2}+\epsilon)D-1$ to leaves are colored with color 1, then we color the edges between $L_{x_i}$ and $L_{x_j}$ with color 2; if the edges between the vertices at depth $(\frac{1}{2}+\epsilon)D-1$ to leaves are colored with color 2, then we color the edges between $L_{x_i}$ and $L_{x_j}$ with color 1. Recalling that we color edges in $M'$ with color 1, then for every $i\neq j$ the path formed by the two edges $\{x_iy_i\}$, $\{x_jy_j\}$ combining with the path $P_{ij}$ is a proper path connecting $y_i$ and $y_j$. Thus our claim follows.

Now we prove that these $(\frac{\log n}{101})$-ary trees can be constructed successfully w.h.p..

Realize first that every vertex $x$ in $H$ has degree $d_{H}(x) \ge \frac{\log n}{100}-2-2$, since there are two edges incident with $x$ in $C$ and $x$ can be adjacent to at most one small vertex plus $u$ in $G$.

For every $i=1,2,\ldots,s+1$, we build the tree $T_{x_i}$ level by level from $x_i$ to the leaves. Suppose that we are growing the tree $T_{x_j}$ from vertex $w$ at depth $k$ to vertices at depth $k+1$. Note that the construction halts if we cannot expand by the required amount. That is, we cannot find enough neighbors of $w$ in $H$ to add into the tree $T_{x_j}$, since $w$ may point to vertices already in $T_{x_i}$, $i \le j$. We call such edges as \emph{ bad edges emanating from $w$}. We claim that the number of bad edges emanating from $w$ is small. It is easy to get that at any stage, the number of vertices we used to construct trees is less than
\begin{align*}
&(s+1)\cdot (\frac{1}{2}+\epsilon)\frac{\log n}{\log \log n}\left(\frac{\log n}{101}\right)^{(\frac{1}{2}+\epsilon)\frac{\log n}{\log \log n}}\\
&\le (\frac{1}{2}+\epsilon)\frac{\log n}{\log \log n}\cdot n^{\frac{1}{2}+\frac{\epsilon}{2}} \cdot n^{0.1}\\
&\le n^{0.65}.
\end{align*}
For any fixed vertex $w$, the bad edges from $w$ is stochastically dominated by the random variable $X \sim Bin(n^{0.65},p)$. Thus,
\begin{align*}
&Pr[there\  are\ at\ least\ 10\ bad\ edges\ emanating \ from\ w]\\
&\le Pr[X\ge 10] \le \binom{n^{0.65}}{10}p^{10}\\
&\le \left(\frac{en^{0.65}}{10}\cdot \frac{\log n+\alpha(n)}{n}\right)^{10}\\
&\le (n^{-0.34})^{10}\\
&=n^{-3.4}.
\end{align*}
Using the Union Bound taking over all vertices, we have that with probability at least $1-n^{-2.4}$, any current vertex $w$ has at most 9 bad edges emanating from it. Therefore, there are at least $\frac{\log n}{100}-4-9-1 \ge \frac{\log n}{101}$ neighbors of $w$ in $H$ that can be used to continue our construction of $T_{x_j}$. Hence, w.h.p. we can successfully build such $\frac{\log n}{101}$-ary trees we required. The proof is thus complete.
\qed\\

\section{Proper connection number of general graphs}

In this section, we use a different method to derive an upper bound for the proper connection number of general graphs.
\begin{thm}\label{lem0}
Let $G=(V,E)$ be a graph. If there are two connected spanning subgraphs $G_1 = (V, E_1)$ and $G_2 = (V, E_2)$ of $G$ such that $|E_1 \cap E_2| \le t$. Then $pc(G) \le t+4$.
\end{thm}
\begin{proof}
We first color the edges in $E_1\cap E_2$ with distinct colors. For the remaining edges in $E_1\cup E_2$, we color them with four new colors different from the colors appeared in $E_1\cap E_2$. We use 1, 2, 3, 4 to denote those new colors. For any two vertices $u$ and $v$, denote the distance between them in $G_i$ by $d_i(u,v)$, where $i=1,2$. Take an arbitrary vertex $x\in V$, define the vertex sets $U_j=\{y\in V:\ d_1(x,y)=j\}$, $W_j=\{y\in V:\ d_2(x,y)=j\}$. Color the edges between $U_{2k-1}$ and $U_{2k}$ with color 1, color the edges between $U_{2k}$ and $U_{2k+1}$ with color 2, for $1\le k\le \lfloor\frac{1}{2}diam(G_1)\rfloor$. Similarly, color the edges between $W_{2\ell-1}$ and $W_{2\ell}$ with color 3, color the edges between $W_{2\ell}$ and $W_{2\ell+1}$ with color 4, for $1\le \ell \le \lfloor\frac{1}{2}diam(G_2)\rfloor$. For the edges in $E\backslash (E_1\cup E_2)$, we can color them with any colors appeared before. Clearly, this coloring uses at most $t+4$ colors.

Now we verify that this edge-coloring is a proper-path coloring of $G$. Let $u$, $v$ be any two vertices in $V$. Choose a shortest $u$-$x$ path $P_1$ in $G_1$, and a shortest $v$-$x$ path $P_2$ in $G_2$. Note that $P_1$ and $P_2$ are proper paths. If they are edge-disjoint, then $P_1\cup P_2$ is a proper path connecting $u$ and $v$. Otherwise, $P_1$ and $P_2$ intersect in edges in $E_1\cap E_2$. We can go from $u$ along $P_1$ till the first common edges, then turn to $P_2$ to reach $v$.
\end{proof}

If $G$ has two edge-disjoint connected spanning subgraphs, then we have $|E_1 \cap E_2|=0$, and therefore, $pc(G)\le 4$. The following result is a straightforward consequence of Theorem \ref{lem0}.

\begin{cor}\label{cor1}
If $G$ has two edge-disjoint spanning trees, then $pc(G)\le4$.
\end{cor}

We remark that we cannot apply Corollary \ref{cor1} to random graph $G(n,p)$ when $p$ is not large enough. It is shown (see \cite{B}) that if $p=\frac{\log n+ \omega_n}{n},$
where $\omega_n\rightarrow \infty$ and $\omega_n\le \log\log\log n$,
then w.h.p. $G(n,p)$ has the minimum degree 1. Therefore, $G(n,p)$ does not have two edge-disjoint spanning trees.\\

\noindent\textbf{Acknowledgement.} We are grateful to Dr. Asaf Ferber for his suggestion which helped to improve our early result $pc(G(n,p))\le 4$ into $pc(G(n,p))\le 3$. Although his result is not the best, however, the structural properties of random graphs he suggested are very helpful to get our best result $pc(G(n,p))\le 2$.

\end{document}